\documentclass[a4paper,12pt]{amsart}

\usepackage{amssymb, amsmath, amsthm, amsfonts}
\usepackage{extsizes}
\usepackage{hyperref}
\usepackage{xcolor,graphicx}
\usepackage{hyperref}
\usepackage{graphicx,color} 

\usepackage{url} 

\newtheorem{theorem}{Theorem}[section]
\newtheorem{corollary}[theorem]{Corollary}
\newtheorem{proposition}[theorem]{Proposition}
\newtheorem{lemma}[theorem]{Lemma}

\theoremstyle{definition}
\newtheorem{definition}[theorem]{Definition}
\newtheorem{example}[theorem]{Example}
\newtheorem{remark}[theorem]{Remark}

\def\r{\mathbb R}
\def\l{\mathbb L}
\def\h{\mathbb H}
\def\m{\mathbb M}

\def\R{\mathcal{ R}}

\def\s{\mathbb S}

\begin{document}

\title[Newton's minimal resistance problem on Riemannian surfaces]{Generalization of Newton's minimal resistance problem to Riemannian surfaces}
\author{Rafael L\'opez}
\address{Department of Geometry and Topology. University of Granada.   18071 Granada, Spain}
\email{rcamino@ugr.es}
\begin{abstract}
We extend Newton's problem of minimal resistance to Riemannian surfaces endowed with a geodesic     coordinate system, which includes the two-dimensional space forms such as the sphere  and the hyperbolic plane. Assuming that the fluid particles flow along radial geodesics, we derive the resistance functional and prove that its smooth extremals are the loxodromes of the surface. Furthermore, we analyze the constrained minimization problem, establishing the absence of strong local minima for smooth extremals, and   characterizing their global minimizers.
 \end{abstract}

\maketitle

\section{Introduction and formulation of Newton's problem}

The problem of minimal resistance was proposed by Newton in his {\it Principia Mathematica}. Roughly speaking, the  problem consists of finding the shape of a solid body in Euclidean space $\r^3$ that offers minimal resistance while moving with constant velocity through a homogeneous fluid flowing parallel to a fixed direction. The literature on this problem is extensive. From a mathematical viewpoint, we refer the reader to the surveys \cite{bu,bk} and references therein. Different many assumptions have been considered: types of bodies depending on convexity or symmetry \cite{bfk,bfk2,ch1,ch2}, types of collision \cite{ch1,py1,py2} and change on the physical conditions \cite{hkv,lo1,lo2}

In the two-dimensional version of the problem,   Plakhov, Silva and Torres investigated the problem when the ambient space is the Euclidean plane $\r^2$ and the direction of the fluid is parallel to a fixed direction  \cite{pt1,pt2,st,tp}. In $(x,y)$ coordinates of $\r^2$, if the   direction is   $(0,1)$, Newton's problem consists of finding minimizers of the functional:
\begin{equation}\label{eq0}
y\mapsto \int_0^{R}\frac{1}{1+y'(x)^2}\, dx,
\end{equation}
where $y\colon [0,R]\to \r$ is a piecewise $C^1$ function. Fixing $H>0$ and assuming boundary conditions $y(0) = 0$ and $y(R) = H$, it was proved that the problem has no global solution in general. However, if the admissible functions are restricted such that $y'(x) \geq 0$, then the problem has a unique solution $y(x) =\frac{H}{R}$ if $R < H$, and infinitely many solutions if $R\geq H$ \cite{st}.  
 
In this paper, we extend Newton's problem of minimal resistance to Riemannian surfaces 
$(\Sigma, ds^2)$ endowed with a geodesic     coordinate system $(u,v)$, where the metric 
takes the local form $ds^2 = du^2 + f(u)^2 dv^2$ and the parametric curves $v=\text{constant}$ are geodesics. A natural motivation for our study 
are the two-dimensional space forms $\mathbb{M}^2(c)$ with constant curvature 
$c \in \{1, 0, -1\}$, namely the sphere $\mathbb{S}^2$, the Euclidean plane $\mathbb{R}^2$, and the hyperbolic plane $\mathbb{H}^2$.

 The  difficulty in extending Newton's problem to curved spaces is the physical interpretation of a `moving body' and the definition of a `fixed direction'. To overcome this, we adopt the kinematically equivalent viewpoint that the object $\mathcal{B}$ is fixed within the surface, and the homogeneous fluid moves around it along the integral curves of a specific   flow.  For the Riemannian surface $(\Sigma,ds^2)$, we will assume that the particles of the fluid travel with velocity $\vec{v}_i$ along the integral curves of the vector field $-\partial_u$.  This distinguishes it from the classical problem in $\r^2$ where the particles move parallel to  a fixed direction of $\r^2$  \cite{st}.   Recently, the author generalized Newton's minimal resistance problem on $\r^3$ assuming that the fluid is radial \cite{lo2}. 

We assume that    $(\Sigma,ds^2)$ is a homogeneous medium of particles and that these particles move along  the vector field $-\partial_u$. Consider a two-dimensional object $\mathcal{B}\subset\Sigma$ whose boundary is a piecewise $C^1$-curve $\gamma$. Then, the particles of $\Sigma$ collide with $\mathcal{B}$ along its boundary $\gamma$ and we will assume that the particles, after the collision, do not hit against $\mathcal{B}$ again  (Single Impact Condition). The choice of $-\partial_u$ is justified by physical and geometric reasons. First,   the speed of the particles is constant because $|\partial_u|=1$ and this preserves the homogeneity of the fluid. Furthermore, the choice of $-\partial_u$ ensures that the particles follow geodesic trajectories of the ambient space $\Sigma$.   

The main result of this paper establishes a   classification of the global minimizers for  Newton's resistance problem, which can be summarized as follows.

\begin{quote}
 {\it Let  $0<u_0<u_1$ and $L=\int_{u_0}^{u_1}\frac{1}{f(u)}\, du$. Let $A=(u_0, v_0)$ and $B=(u_1, v_1)$ be two points  on $(\Sigma,ds^2)$  and let $\Delta v = v_1 - v_0$ be the angular amplitude between the two points. In the class of admissible piecewise monotonic $C^1$-curves  connecting $A$ and $B$, the following holds:

\begin{enumerate}
    \item  If $\Delta v \le L$, then the unique absolute minimizer is a (smooth) loxodrome connecting $A$ and $B$. 
        \item If $\Delta v > L$, then  the absolute minimum is attained by a truncated loxodrome, consisting of a loxodromic arc of exactly angular amplitude $L$ followed by a flat, constant-radius circular arc.
\end{enumerate}}
\end{quote}

The paper is organized as follows. In Section \ref{s2}, we derive the resistance functional. 
In Section \ref{s3}, we obtain the Euler-Lagrange equations and prove in Theorem \ref{t1}
 that the smooth extremals are the loxodromes of the surface, that is, curves that form a constant angle with the parametric curves $v=\text{constant}$. Moreover, we show that 
 for an extremal to be a local minimizer, its attack angle must be less than $\pi/3$ (Proposition \ref{pr37}), 
 a universal constant independent of the   curvature of the surface. In Section \ref{s4}, we investigate the constrained minimization problem and the concept of truncated loxodromes.  In Section \ref{s5}, we find  necessary and sufficient conditions for local minima, proving that loxodromes can provide weak local minima but cannot be strong local minima (Corollary \ref{co44}).   Finally, in Section \ref{s6}, we   solve the global minimization problem, demonstrating that the nature of the optimal profile depends   on the angular amplitude between the two points,   as previously  stated.  

\section{Derivation of the resistance functional}\label{s2}

Let $(\Sigma,ds^2)$ be a Riemannian surface and consider $\Psi\colon\Omega\subset\r^2\to \Sigma$ a parametrization of $\Sigma$, where $\Omega=I\times J$ is a rectangular domain of $\r^2$. Let $(u,v)$ denote the coordinates of $\Omega$. Throughout this paper, we will assume that the metric $ds^2$ can be locally written as 
\begin{equation}\label{ds}
ds^2=du^2+f(u)^2dv^2,
\end{equation}
where $f=f(u)$ is a smooth positive function. In a more general context, if $ds^2=du^2+G(u,v)^2dv^2$ and the parametric curves $v=\textrm{constant}$ 
are geodesics, the     coordinate system is called a geodesic coordinate system \cite[p. 136]{str}. This is the case for normal coordinates in a Riemannian surface centered at a fixed point and defined by radial geodesics starting from that point. In our framework, we have $G(u,v)=f(u)$, but we do not necessarily require  that the parametric curves $v=\text{constant}$ are geodesics of $\Sigma$. However, from a physical viewpoint, the fact that these curves are geodesics gives relevance and meaning to our model. For example, this is the situation in space forms, as we will see below.  To simplify the notation, the parametric curves 
$u=\text{constant}$ and $v=\text{constant}$ will be called parallels and meridians, respectively.

 From now on, we will identify $\Sigma$ with its parametrization (or immersion) $\Psi$. Throughout this paper, the points of the surface are defined by their (unique) geodesic coordinates $(u,v)$ instead of their images under $\Psi$. We point out that if we fix two points on $ \Psi(\Omega)$, there may exist many curves connecting the points. For example, in the space forms, the immersion $\Psi$ is $2\pi$-periodic in the coordinate $v$. In such a case, the point $B$ can be represented by the coordinates $(u_1, v_1 + 2\pi n)$ for any $n \in \mathbb{Z}$.

A first example of geodesic coordinates are the standard (latitude/longitude) parametrization of  a surface of revolution in $\r^3$. A second example is given by  the two-dimensional space forms $\m^2(c)$ with of curvature $c\in\{1,0,-1\}$, which we briefly recall. 

\begin{enumerate}
 \item Sphere $\s^2$ ($c=1$). Let $\s^2=\{(x,y,z)\in\r^3 \colon x^2+y^2+z^2=1\}$ be endowed with the Euclidean metric $\langle,\rangle$ from $\r^3$. We parametrize $\s^2$ by $\Psi\colon [-\frac{\pi}{2},\frac{\pi}{2}]\times \r\to\s^2$ given by 
 $$ \Psi(u,v) = (\cos u\cos v, \cos u\sin v, \sin u). $$
  
 \item Euclidean plane $\r^2$ ($c=0$). We consider polar coordinates 
 $$\Psi(u,v) = (u \cos v, u \sin v),$$
 where $(u,v)\in [0,\infty)\times [0,2\pi]$. 
 
 \item Hyperbolic plane $\h^2$ ($c=-1$). We use the hyperboloid model $\h^2 = \{(x,y,z) \in \r^{2,1} \colon x^2+y^2-z^2=-1, z>0\}$, whose metric is induced from the Lorentzian space $\r^{2,1}$, that is, $\langle,\rangle=dx^2+dy^2-dz^2$. We parametrize $\h^2$ by $\Psi\colon [0,\infty)\times\r\to\h^2$, where  
 $$\Psi(u,v)=(\sinh u\cos v,\sinh u\sin v,\cosh u).$$
 \end{enumerate}

In the space forms $\m^2(c)$, the function $f(u)$ of the metric \eqref{ds} is
\begin{equation}\label{ff}
f(u)=\left\{ \begin{array}{ll}
 \cos u & \quad \s^2,\\
 u &\quad\r^2,\\
 \sinh u &\quad\h^2.
 \end{array}\right.
 \end{equation}
 
 On $(\Sigma,ds^2)$, a basis for the tangent space is $\{\Psi_u, \Psi_v\}$. The coefficients of the first fundamental form are \begin{equation*}
\langle \Psi_u, \Psi_u \rangle = 1, \quad \langle \Psi_v, \Psi_v \rangle = f(u)^2, \quad \langle \Psi_u, \Psi_v \rangle = 0.
\end{equation*}
The trajectories of the particles are tangent to the meridians and their velocities are
\begin{equation*}
\vec{v}_i =-\partial_u= -\Psi_u.
\end{equation*}
 As stated previously, we assume that the fluid particles move along the meridians of $\Sigma$.   Notice that in $\m^2(c)$, these curves are geodesics flowing towards (or away from) the pole in $\s^2$ and $\h^2$ or the origin in $\r^2$. Assume that the boundary of $\mathcal{B}$ is a piecewise $C^1$-curve parametrized by $\gamma(t) = \Psi(u(t), v(t))$, where $u(t)$ and $v(t)$ are piecewise $C^1$ functions defined in some interval of $\r$. The tangent vector to the curve is $ \gamma'(t) = u' \Psi_u +v' \Psi_v$, where $u' = \frac{du}{dt}$ and $v' = \frac{dv}{dt}$. Thus, the unit tangent vector $T$ is
\begin{equation*}
T = \frac{1}{\sqrt{u'^2 +v'^2 f(u)^2}} (u' \Psi_u +v' \Psi_v).
\end{equation*}
To compute the resistance, we need the unit outward normal vector $N$ to $\gamma$, which is
\begin{equation*}
N = \frac{1}{\sqrt{u'^2 +v'^2 f(u)^2}} \left(v' f(u) \Psi_u - \frac{u'}{f(u)} \Psi_v \right).
\end{equation*}
We now analyze the impact of the particles on the boundary curve $\gamma$.  We assume that the collision of the fluid particles with  $\gamma$ is perfectly elastic and that if the particles hit   $\gamma$, they do so only once. At each point $\gamma(s)$, we decompose $\vec{v}_i$ as 
\begin{equation}\label{vi}
\vec{v}_i = \cos\theta T-\sin\theta N,
\end{equation}
for a certain (non-constant) angle. Upon impact, the tangential component of the particle's velocity is preserved, while the normal component is reversed. Thus, the final velocity $\vec{v}_f$ after the collision is 
 $$\vec{v}_f = \cos\theta T+\sin\theta N.$$
 Thus, we can also write $\vec{v}_f$ as $\vec{v}_f = \vec{v}_i - 2\langle \vec{v}_i, N \rangle N$.  The momentum transferred from a particle to the body is proportional to the change in velocity $\vec{v}_f - \vec{v}_i = -2\langle \vec{v}_i, N \rangle N$. The resistance experienced by the body is defined as the component of this momentum transfer in the direction opposing the flow ($-\vec{v}_i$). Therefore,   using \eqref{vi}, the resistance density   is  
\begin{equation}\label{r2}
d\R = \langle \vec{v}_f - \vec{v}_i, -\vec{v}_i \rangle = 2\langle \vec{v}_i, N \rangle^2= \frac{2v'^2f(u)^2}{u'^2 +v'^2 f(u)^2}.
\end{equation}
In the   $(u,v)$ coordinate system, assuming that the mass is conserved, the fluid density $\rho(u)$ must 
satisfy the equation $\text{div}(\rho \vec{v}_i)=-\text{div}(\rho\partial_u) = 0$. The expression for the divergence of a vector field $X$ in coordinates $X^1\Psi_u+X^2\Psi_v$ and with respect to the metric \eqref{ds} is 
 $\text{div} X = \frac{1}{\sqrt{g}} \sum_k \frac{\partial}{\partial x^k} \left( \sqrt{g} X^k \right)$, where $g=f^2$. Then $\text{div}(\rho \vec{v}_i)=0$ is equivalent to 
$$\frac{1}{f(u)} \frac{\partial}{\partial u} \big(\rho(u) f(u)\big) = 0.$$
This implies that $\rho(u) f(u)$ 
must be constant along the flow lines. Consequently, the density $\rho(u)$ is proportional to 
$ 1/f(u)$. This affects   the computation of  the total resistance because we must weight the local momentum transfer by the flux of particles 
 striking a differential 
 segment $d\gamma$ of $\gamma$. Because the particles travel along the $u$-curves, 
 the relevant geometric width of the flow intercepted by $d\gamma$ is its orthogonal 
 projection along the $v$-curves, 
 which is $f(u) dv=f(u)v'(t)\, dt$. Thus, the number of particles 
 impacting this segment per unit time is proportional to 
 $$ \rho(u) f(u) dv =\left( \frac{1}{f(u)} \right) f(u) dv = dv = v' dt.$$ 
 Multiplying \eqref{r2} by $v' dt$, dropping the constant $2$ and integrating, 
 we obtain the expression of the resistance.

\begin{definition}
 Let $\gamma(t) = \Psi(u(t), v(t))$, $t \in [a,b]$, be a (regular) curve in $(\Sigma,ds^2)$. The total resistance 
 functional of $\gamma$ with respect to the fluid of particles with velocity $-\partial_u$ is defined as
\begin{equation}\label{r1}
\R[\gamma] = \int_a^b \frac{v'^3 f(u)^2}{u'^2 +v'^2 f(u)^2} \, dt.
\end{equation}
\end{definition}

 \begin{remark} For the spaces  $\s^2 $ and in $\h^2$, one might consider them as   objects embedded in a higher-dimensional space: the Euclidean space $\r^3$ for $\s^2$ or the Lorentzian space $\r^{2,1}$ for $\h^2$. From this viewpoint, an alternative choice for the   velocity of the particles is the tangential projection $(\partial _z)^T$ of the constant ambient vector field $\partial_z$. However, this vector field does not have constant speed in $\s^2$ nor in $\h^2$. While $(\partial _z)^T$ is geometrically intuitive from an observer outside $\m^2(c)$, a variable speed field 
suggests the existence of external forces acting upon the particles, which 
contradicts the minimal resistance framework, where resistance arises solely from geometric impact. 
\end{remark}
 
 \section{Extremals of the functional}\label{s3}

In this section, we prove that the extremals for the resistance functional \eqref{r1} can be obtained explicitly, by showing that they   are the loxodromes of the space. The first step is to  deduce the Euler-Lagrange equations for the energy \eqref{r1}. The extremals of the functional    \eqref{r1} are smooth solutions of these Euler-Lagrange equations
$$\frac{\partial \mathcal{L}}{\partial u} - \frac{d}{dt} \left( \frac{\partial \mathcal{L}}{\partial u'} \right)=0, \quad \frac{\partial \mathcal{L}}{\partial v} - \frac{d}{dt} \left( \frac{\partial \mathcal{L}}{\partial v'} \right)=0.$$
 The Lagrangian of the resistance functional is $\mathcal{L}(u, u') = \frac{v'^3f^2}{ u'^2+v'^2f^2}$. A straightforward computations leads to  the following result.

 \begin{proposition}  The Euler-Lagrange equations of the resistance functional \eqref{r1} are
 \begin{equation}\label{eq1}
 \begin{split}
 0&= v'^3 \left( 3u'^2 - f(u)^2 v'^2 \right) \left[ f(u) u'' - \frac{f(u) u'}{v'} v'' - f'(u) u'^2 \right] ,\\
 0&= u' v'^2 \left( 3u'^2 - f(u)^2 v'^2 \right) \left[ f(u) u'' - f(u) \frac{u'}{v'} v'' - f'(u) u'^2 \right].
 \end{split}
 \end{equation}
 Moreover, there is a constant $C$ such that
 \begin{equation}\label{c1}
 \frac{v'^2 f(u)^2 (3u'^2 + v'^2 f(u)^2)}{(u'^2 + v'^2 f(u)^2)^2} = C.
 \end{equation}
 \end{proposition}
 
 \begin{proof}
   For \eqref{c1}, note that $\mathcal{L}$ does not depend on $v$. Thus, $\frac{\partial \mathcal{L}}{\partial v'} $ is constant. 
 \end{proof}

From \eqref{eq1}, it follows that the parametric curves are extremals. 

\begin{corollary}
Meridians and parallels are extremals of the resistance functional.
 \end{corollary}
 
 The value of the resistance functional     for the parametric curves is easy to calculate. 
 \begin{enumerate}
 \item For the meridians $v' = 0$, we have $\R[\gamma] = 0$. Physically, the meridian is aligned with the flow lines, intercepting no particles and thus offering zero resistance.
 \item For the parallels $u' = 0$, we have $\R[\gamma] = \int_{t_0}^{t_1} v' \, dt = v(t_1)-v(t_0)$. That is, the resistance depends only on the angular amplitude between  the endpoints of the parallel.
\end{enumerate}

Meridians and parallels are solutions to two simple problems of minimization and maximization, respectively, which justifies the fact that both curves are extremals of the resistance functional. The first result is a consequence of the fact that the resistance is non-negative and that for meridians, the resistance is $0$.

\begin{corollary}
Consider two points $A = (u_0, v_0)$ and $B = (u_0, v_1)$ on $\Sigma$ in geodesic coordinates. If $v_0 = v_1$, then the meridian segment between both points is the global minimizer of the resistance functional among all piecewise $C^1$-curves connecting $A$ and $B$. 
\end{corollary}

We now prove that parallels are maximizers of the resistance problem. 

\begin{corollary}\label{co34}
Consider two points $A = (u_0, v_0)$ and $B = (u_1, v_1)$ on $\Sigma$ in geodesic coordinates, with $u_0 =u_1$. Then the segment of the parallel between these points is the global maximizer of the resistance functional among all piecewise $C^1$-curves connecting $A$ and $B$. 
\end{corollary}

\begin{proof} The segment of the parallel connecting $A$ and $B$ is $\gamma_{par}(t)=\Psi(u_0,t)$, with $t\in [v_0,v_1]$. We know that $\R[\gamma_{par}]=v_1-v_0$. Let $ \gamma(t)=\Psi(u(t),v(t))$, with $t\in [a,b]$, be any piecewise $C^1$-curve such that $v(a)=v_0$ and $v(b)=v_1$. Then 
$$\R[\gamma] = \int_a^b \frac{v'^3 f(u)^2}{u'^2 +v'^2 f(u)^2} \, dt\leq \int_a^b v'(t) \, dt=v_1-v_0=\R[\gamma_{par}].$$
\end{proof}

With the exception of meridians, extremals are never tangent to meridians as they intersect them, as shown by the following result. 

\begin{proposition} If $\gamma$ is an extremal which is tangent to a meridian at some point, then $\gamma$ is a meridian.
\end{proposition}

\begin{proof} 
Suppose $\gamma(t)=\Psi(u(t),v(t))$. If there is a time $t_0$ such that $\gamma'(t_0)$ is proportional to $\Psi_u$ at $\gamma(t_0)$, then \eqref{c1} implies that the constant is $C=0$. Thus $v'^2(3u'^2+v'^2f(u)^2)=0$ along $\gamma$. Since the expression in parentheses cannot be $0$ because $\gamma$ is regular, we must have $v'(t)=0$ for all $t$. This proves that $\gamma$ is a meridian.
\end{proof}

From this result, any extremal, other than a meridian, can be parametrized as $u=u(v)$. This simplifies the Euler-Lagrange equation to
\begin{equation}\label{EL}
 \left( f(u)^2 - 3u'^2 \right) \left( f(u)u'' - f'(u)u'^2 \right)= 0.
\end{equation}
Also, the resistance functional \eqref{r1} becomes 
\begin{equation}\label{r11}
\R[u]=\int_{v_0}^{v_1}\frac{f(u)^2}{u'^2+f(u)^2}\, dv.
\end{equation}
To find the extremals, we solve the ODE \eqref{eq1} taking into account that the equation is formed by the product of two factors. 
Note that we can discard the terms $v'^3$ and $u'v'^2$ from the original system \eqref{eq1}  by the regularity of the curve. First, we investigate the solutions of the second factor. 

\begin{theorem} \label{t1}
  The extremals corresponding to the second factor of \eqref{eq1}, that is, 
\begin{equation}\label{f2}
 f(u) u'' - \frac{f(u) u'}{v'} v'' - f'(u) u'^2 =0
 \end{equation}
identically in the domain, are the loxodromes of $\Sigma$, that is, curves that make a constant angle with the meridians, which are given by the ODE
 \begin{equation}\label{lo1}
 u'=kv'f(u),
 \end{equation}
 where $k$ is a positive constant. The resistance of the loxodrome defined in $[v_0,v_1]$ is 
\begin{equation}\label{r5}
\mathcal{R}=\frac{v_1-v_0}{1+k^2}.
\end{equation}
\end{theorem}
 
 \begin{proof}
 If $\gamma$ is a parallel or a meridian, the result is trivial. Suppose $u'\neq 0$ at some point $t_0$, meaning $u'(t)\not=0$ in a neighborhood $(t_0-\delta,t_0+\delta)$ of $t_0$. Working on this interval, we divide \eqref{f2} by $u'f(u)$ to obtain 
 $$ \frac{u''}{u'} -\frac{v''}{v'}- \frac{f'(u)}{f(u)} u' = 0.$$
 We can integrate this, deducing that there is a constant $k>0$ such that $ u'=kv'f(u)$. This proves \eqref{lo1}. Moreover,  the angle $\theta$ that $\gamma$ makes with the meridians is defined by 
 $$\cos\theta=\frac{\langle\gamma',\Psi_u\rangle}{|\gamma'|}=\frac{\langle u'\Psi_u+v'\Psi_v,\Psi_u\rangle}{\sqrt{u'^2+v'^2f^2}}= \frac{u'}{|v'|f\sqrt{1+k^2}}=\pm \frac{k}{\sqrt{1+k^2}}.$$
 Thus $\theta$ is constant. The formula of the resistance is a simple integration of \eqref{r1} using \eqref{lo1}.
\end{proof}

Theorem \ref{t1}  generalizes what happens in the classical Newton's problem. Indeed, in \cite{st} it was proved that the extremals of the functional \eqref{eq0} are the lines $y(x)=c\cdot x$ with $c\in\r$. All these lines  make a constant angle with the direction $(0,1)$ of the flow. 

We now particularize Theorem \ref{t1} in   the space forms $\m^2(c)$. In this   case, the loxodromes can be obtained explicitly by a simple integration of \eqref{lo1}, where $f$ is given by \eqref{ff}. See Fig  \ref{fig1}. Without loss of generality, we suppose that the extremals are expressed as  $u=u(v)$.

\begin{corollary}
In the space forms $\m^2(c)$, the extremals $\gamma(v)=\Psi(u(v),v)$ corresponding to the second factor of \eqref{eq1} are the following:
\begin{enumerate} 
\item Case $\s^2$. Spherical loxodromes given by $u(v) = 2 \arctan(A e^{kv}) - \frac{\pi}{2}$, with $A\in \r$. 
 \item Case $\r^2$. Logarithmic spirals $u(v) = A e^{kv} $, with $A\in\r$. 
\item Case $\h^2$. Hyperbolic loxodromes given by $ u(v)= 2 \operatorname{arctanh}(A e^{kv})$, with $A\in\r$. 
\end{enumerate}
\end{corollary}

 \begin{figure}[h!t]
\centering
\includegraphics[width=0.27\linewidth]{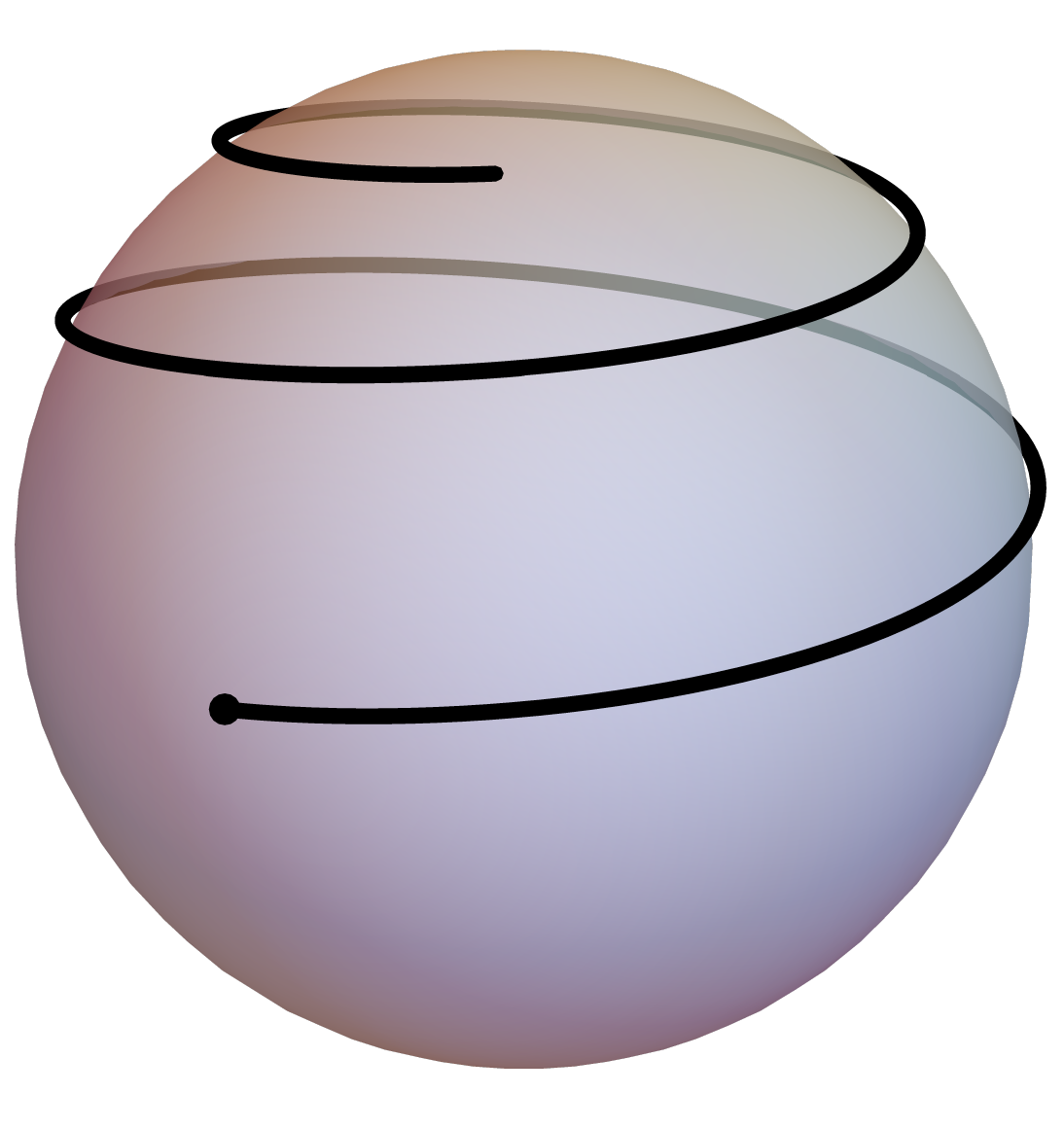}\quad \includegraphics[width=0.3\linewidth]{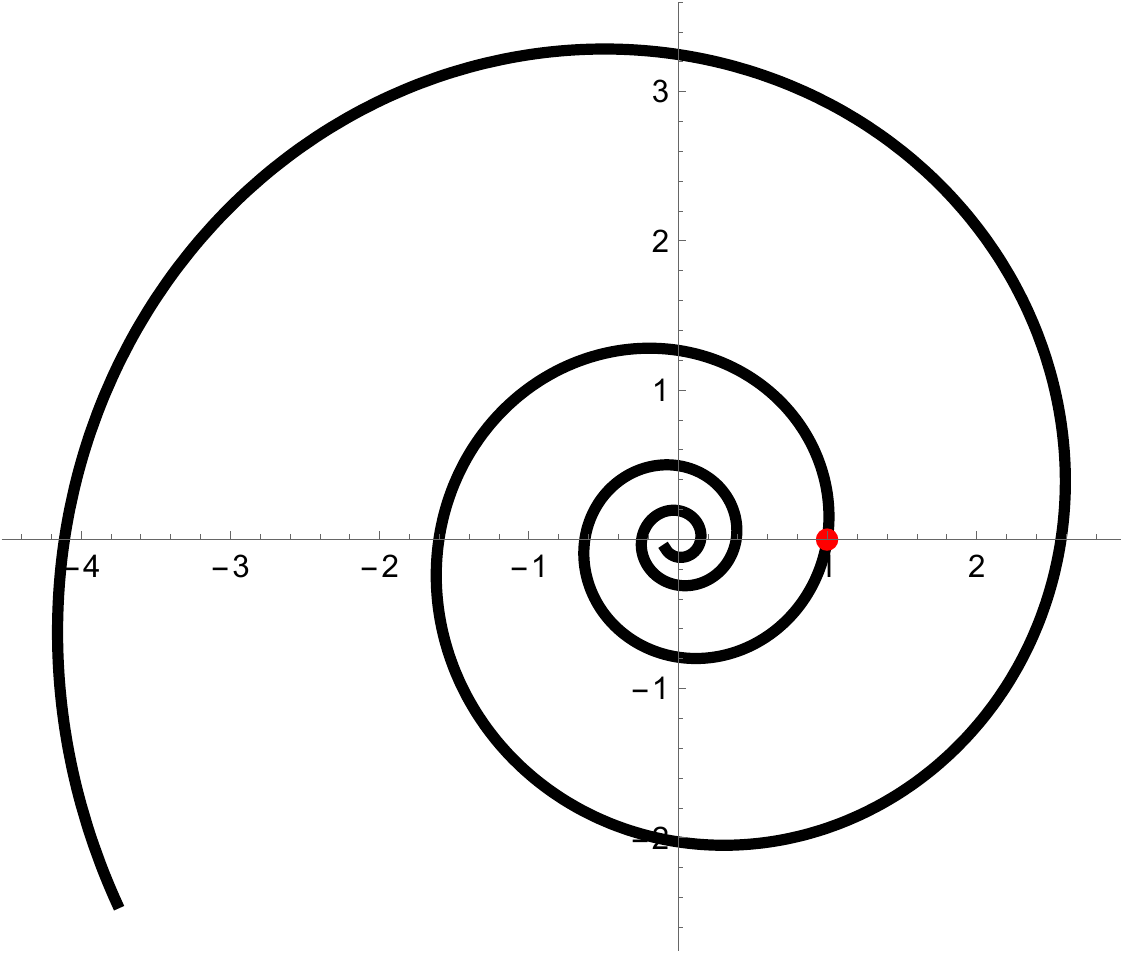}\quad \includegraphics[width=0.3\linewidth]{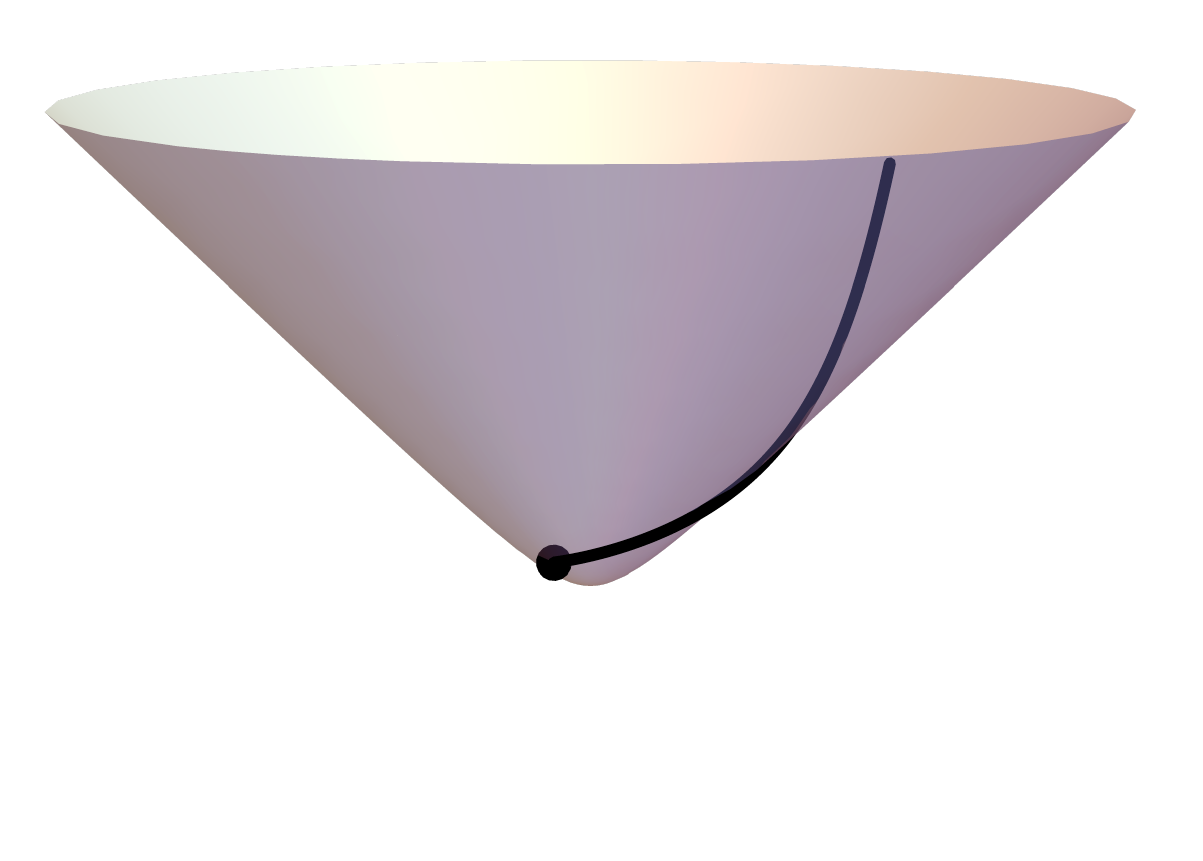}
\caption{Loxodromes in the three space-forms: sphere (left), plane (middle) and hyperbolic plane (right).}
\label{fig1}
\end{figure}

We now analyze the solutions of \eqref{eq1} coming from the first factor.

\begin{theorem}\label{t36}
  Extremals corresponding to the first factor of \eqref{eq1}, that is, 
$$ 3u'^2 - f(u)^2 v'^2 =0$$
identically in the domain, are the loxodromes of $\Sigma$ with angle $\pi/3$. 
\end{theorem}
 
 \begin{proof}
If $ 3u'^2 - f(u)^2 v'^2 =0$ identically, then 
$$u'=\frac{1}{\sqrt{3}}v'f(u).$$
 Following the proof of the previous theorem, the angle $\theta$ that $\gamma$ makes with the meridians is $\pi/3$ because $\cos\theta= 1/2$. 
\end{proof}

The value $1/\sqrt{3}$ in Theorem \ref{t36} appears in the classical problem  of $\r^3$ (also in the two-dimensional version \cite{st}).  In Newton's resistance problem of $\r^3$ with velocity of the flow given by the direction $(0,0,1)$, suppose that   $\mathcal{B}$ is a rotationally symmetric about the $z$-axis. If  the boundary surface is    parametrized in radial coordinates $u=u(r)$, then the maximal domain of $u(r)$ is defined just until the slope of $u$ is $u'(r)=1/\sqrt{3}$.

Following up on  the value $1/\sqrt{3}$, we see that this number is not arbitrary in the following 
sense. In the calculus of variations, a necessary condition for an extremal curve to yield a local 
minimum of a functional is the Legendre condition, which requires the second partial derivative 
of the Lagrangian with respect to $u'$ to be strictly positive. We now provide a relationship between 
this condition and the constant $1/\sqrt{3}$.

\begin{proposition}[Legendre condition] \label{pr37}
 If an extremal curve $\gamma(t)=\Psi(u(t),t)$ of the energy \eqref{r11} is a weak local minimum, then $3u'^2 > f(u)^2$. 
\end{proposition}

\begin{proof}
 A simple computation yields
\begin{equation*}
\frac{\partial^2 \mathcal{L}}{\partial u'^2} = \frac{2f(u)^2 \left( 3u'^2 - f(u)^2 \right)}{\left( f(u)^2 + u'^2 \right)^3}.
\end{equation*}
Thus $ \frac{\partial^2 \mathcal{L}}{\partial u'^2} >0$ if and only if $3u'^2 >f(u)^2$ everywhere. 
\end{proof}

From the physical viewpoint, the critical value $1/\sqrt{3}$  indicates that   minimizers of 
Newton's problem must cross the meridians at an angle   strictly less than $\frac{\pi}{3}$. 
We point out that this bound appears as a universal constant independent of the curvature of the Riemannian surface $\Sigma$.

 Another well-known result in Newtonian aerodynamics states that in $\mathbb{R}^3$, the resistance of a sphere is exactly $1/2$ of the resistance of a flat disk of the same radius. We now show a similar result in the two-dimensional setting. Consider the case of the plane $\r^2$. The flat disk is now replaced by a circular arc   $C_{arc}$ centered at the origin because the flow is perpendicular to this arc. The equivalent to the sphere is a piece of a circle $C_{tan}$ that is tangent to the cone determined by $C_{arc}$. In fact, there are two circles tangent to this cone.
 
 \begin{proposition} Let $C_{arc}$ be a circular arc centered at the origin with angular sector $v\in [-v_0,v_0]$, with $0<v_0<\pi$, and let $C_{tan}$ be one of the two circles tangent to the rays $v=\pm v_0$. Then $\R[C_{arc}]=2v_0$ and  
\begin{equation}\label{ec}
\R[C_{tan}]=\cot v_0 + v_0(1 - \cot^2 v_0).
\end{equation}
 Moreover, the ratio of resistance between $C_{tan}$ and $C_{arc}$ satisfies 
$$
 \lim_{v_0 \to 0} \frac{\R[C_{tang}]}{\R[C_{arc}]} = \frac{2}{3},\quad \lim_{v_0 \to \pi/2} \frac{\R[C_{tang}]}{\R[C_{arc}]} = \frac{1}{2}$$
 \end{proposition}
 
 \begin{proof}
 Let $R$ denote the radius of   $C_{arc}$, that is, $u(v)=R$, where $v\in [-v_0,v_0]$. We know that the resistance of this curve, which is a piece of a parallel, is the angular amplitude (see \eqref{r5}, where $k=0$). We now parametrize the two circles tangent to the cone determined by $C_{arc}$ and centered on the $x$-axis. A circle tangent to the rays $v = \pm v_0$ and centered at $(u_c, 0)$ must have a radius $r = u_c \sin v_0$. Its polar equation is $u^2 - 2u u_c \cos v + u_c^2 \cos^2 v_0 = 0$. Thus, 
\begin{equation}
u(v) = u_c \left( \cos v \pm \sqrt{\cos^2 v - \cos^2 v_0} \right) .
\end{equation}
The computation of the energy shows that the value does not depend on the sign taken in $ \pm \sqrt{\cos^2 v - \cos^2 v_0}$. A direct computation gives \eqref{ec}. The limits of the ratios of resistance are easily derived.
 
\end{proof}

\section{The constraint minimization problem and truncated extremals}\label{s4}

In this section, we present an initial approach to the minimization problem. The first issue is the definition of the space of admissible functions. Suppose initially that $\mathcal{B}$ can have an arbitrary height. We first examine this case when   the ambient space is $\r^2$. Assume that the boundary curve $\gamma$ is expressed in polar coordinates as $u=u(v)$. We denote by $\mathcal{R}[u]$ the resistance of $\gamma$. We fix two points of $\r^2$ in polar coordinates, namely, $A=(u_0,v_0)$ and $B=(u_1,v_1)$, with $u_0,u_1>0$ and $v_0<v_1$. Suppose that the second point $(u_1,v_1)$ is not prescribed (unconstrained  minimization problem) and let 
$$\mathcal{C}_0=\{ u \in PC^1([v_0,v_1],\r^+)\colon u(v_0)=u_0\}.$$
be the space of admissible functions. We see that the minimization problem in $\mathcal{C}_0$ has no   solution because we can find functions $u\in\mathcal{C}_0$ whose resistance is arbitrarily close to $0$. An initial  attempt is to mirror the classical setting and consider segments whose vertices are arbitrarily high.   To simplify the calculations, suppose that $A=(1,0)$ and $B=(u_1,\pi/2)$. In polar coordinates,  the segment joining $A$ and $B$ is given by 
$$u(v)=\frac{u_1}{u_1\cos v+\sin v},\quad v\in [0,\frac{\pi}{2}].$$
Its resistance is 
$$\mathcal{R}[u]=\frac{\pi}{4}+\frac{u_1}{1+u_1^2}.$$
If $u_1\to\infty$ (arbitrary height), then $\mathcal{R}[u]\to \pi/4$. This shows that the choice of segments, following a similar approach as in the classical Newton's problem, is not suitable for our purposes. In fact, this is expected since the segments are not well-suited to the radial geometry of the problem. 

Instead of segments, we consider the smooth extremals of $\mathcal{R}$, which are the loxodromes of the surface $\Sigma$. In geodesic coordinates, if we let $u=u_k(v)$ denote the solution of \eqref{lo1},   the resistance is given by \eqref{r5}, that is,   
$$\mathcal{R}[u_k] = \frac{v_1-v_0}{1+k^2},$$
which goes to $0$ as $k \to \infty$. This justifies the need to prescribe fixed   radii $u_0$ and $u_1$ at the endpoints $v=v_0$ and $v=v_1$ to obtain a well-posed minimization problem.  

Let us fix the two points $A$ and $B$ and consider admissible functions with radial values bounded between $u_0$ and $u_1$ ($u_0 < u_1$). Let 
$$\mathcal{C}_{0,1} = \{ u \in PC^1([v_0,v_1],\r^+)\colon u_0 \leq u(v) \leq u_1, u(v_0)=u_0, u(v_1)=u_1\}$$
be the space of admissible functions. From now on, we introduce the notation 
$$\Delta v= v_1 - v_0.$$

We show with an example that the minimization problem in $\mathcal{C}_{0,1}$ has no solution without some type of monotonicity constraint. By defining a sequence of functions $w_n \in \mathcal{C}_{0,1}$ whose radial value oscillates rapidly between $u_0$ and $u_1$, we see that the resistance tends to $0$ as the number of oscillations grows: see \cite{le} for the classical problem.

\begin{example}
To simplify the calculations, fix $v_0=0$ and $v_1=\pi/2$. Let $L = \int_{u_0}^{u_1} \frac{1}{f(t)} \, dt$. We define the sequence of functions $\{w_m\}_{m\in\mathbb{N}}$ in $\mathcal{C}_{0,1}$ as the unique solutions to the   initial value problem
\begin{equation*}
\begin{cases} 
w_m'(v) = m L f(w_m(v)) \sin(2mv), \\
w_m(0) = u_0.
\end{cases}
\end{equation*}
By separation of variables, we have 
$$\int_{u_0}^{w_m(v)} \frac{1}{f(t)} \, dt = \frac{L}{2} (1 - \cos(2mv)).$$
 Since the range of the right-hand side is $[0, L]$, the monotonicity of the integral in the left-hand side ensures that  $u_0 \leq w_m(v) \leq u_1$ for all $v \in [0, \pi/2]$. Moreover, for any odd  $m$, we have $w_m(\pi/2) = u_1$. Thus, we can ensure that $w_m \in \mathcal{C}_{0,1}$. 
 
We now compute  the resistance of $w_m$ using \eqref{r11}. Since  $(w_m')^2 = m^2 L^2 \sin^2(2mv)/f(w_m)^2$,   we obtain
$$\mathcal{R}[w_m] =   \int_0^{\pi/2} \frac{1}{1 + m^2 L^2 \sin^2(2mv)} \, dv.$$
Doing   the change of variables $\theta = 2mv$, we have
$$\mathcal{R}[w_m] = \frac{1}{2m} \int_0^{m\pi} \frac{1}{1 + m^2 L^2 \sin^2 \theta} \, d\theta = \frac{1}{2} \int_0^{\pi} \frac{1}{1 + m^2 L^2 \sin^2 \theta} \, d\theta,$$
where we have used that the integrand is $\pi$-periodic. 
As $m \to \infty$, the integrand converges to zero except when $\sin\theta=0$. Thus, the integrand converges  pointwise to zero  almost   everywhere. By the dominated convergence theorem, we conclude that $\mathcal{R}[w_m] \to 0$ as $m \to \infty$.
\end{example}

 \begin{corollary}
The resistance functional admits no global minimizer in the class of piecewise $C^1$-curves $\mathcal{C}_{0,1}$ connecting $A$ and $B$ and bounded between $u(A)$ and $u(B)$.
\end{corollary}

From this result, we need to restrict the set $\mathcal{C}_{0,1}$.  In order to avoid the oscillations of the admissible functions, we assume that the functions $u=u(v)$ are  monotonically increasing. For $0<u_0<u_1$, define the space of admissible piecewise monotonic  $C^1$ functions as 
$$\mathcal{C}_{0,1}^m=\{ u \in PC^1([v_0,v_1],\r^+)\colon u'\geq 0,   u(v_0)=u_0, u(v_1)=u_1\}.$$ 
The monotonicity condition $u' \geq 0$ ensures that the curve  moves towards increasing radial values,   representing a physical aerodynamic front.  Notice that if the points have the same value of $u$, the set $\mathcal{C}_{0,1}^m$ only has one element, and the problem of minimization is trivial.

 Following   the analogy of the classical problem, we will consider truncated curves, that is curves joining the points $A$ and $B$ that  have exactly one point where 
 smoothness fails at that point \cite{go}. Motivated by the truncated cones in the classical context of $\r^3$ (or in the 
 two-dimensional case described in \cite{st}), we assume that around the endpoint $B$, the 
 curve $\gamma$ offers maximum resistance. This implies that the curve is orthogonal to the radial 
 direction of the flow and thus, the function $u(v)$ is constant, with $u(v)=u_1$ (see also Corollary \ref{co34}).  The second piece of $\gamma$ will be an extremal. 
 We define a truncated loxodrome $\gamma_{trunc}$, which begins as a loxodrome and ends as a circular arc matching the final radius $u_1$. See Figure  \ref{fig2} for the ambient space is $\r^2$. 
 
 First,  let us introduce the change of variable
 $$y(v)=\Phi(u(v))=\int\frac{1}{f(u)}\, du,$$
 which will be used throughout this paper.  Then  
\begin{equation}\label{16}
\mathcal{R}[u] = \int_{v_0}^{v_1} \frac{1}{1 + y'^2} \, dv.
\end{equation}
 For example, the loxodromes \eqref{lo1} can be expressed, after an integration by parts, as $\Phi(u)=k+C$, for some constant $C$. 
In the space forms $\m^2(c)$, we find
 $$\Phi(u)=\left\{\begin{array}{ll}
 \log(\tan\frac{u}{2})& c=1,\\
 \log(u)& c=0,\\
 \log(\tanh\frac{u}{2})&c=-1.\end{array}
 \right.$$
 
 \begin{definition}
We define a truncated loxodrome $\gamma_{trunc}^{v_c}$   connecting $A=(u_0, v_0)$ and $B=(u_1, v_1)$ as the curve $u=u(v)$ given by
\begin{equation*}
\Phi(u(v)) = \begin{cases} 
 k(v - v_0) +\Phi(u_0) & v \in [v_0, v_c] \\
\Phi(u_1) & v \in [v_c, v_1],
\end{cases}
\end{equation*}
 where   continuity at $v_c$ requires that $k$ takes the value 
 \begin{equation}\label{k1}
 k = \frac{\Phi(u_1) - \Phi(u_0)}{v_c - v_0}.
 \end{equation}
\end{definition}
Notice that $\Phi$ is an increasing function because $f>0$. Thus, if we assume that $u_0<u_1$, then the constant $k$ is positive.

\begin{theorem}\label{t51}
Let the boundary points be $A = (u_0, v_0)$ and $B = (u_1, v_1)$ on $\Sigma$, with $u_0 < u_1$ and $v_0 < v_1$. Consider the class of truncated loxodromes $\mathcal{C}_{0,1}^{trunc} = \{\gamma_{trunc}^{v_c} : v_c \in (v_0, v_1)\}$. Let $L = \Phi(u_1) - \Phi(u_0)$. If the angular amplitude satisfies $\Delta v > L$, then the minimum of the resistance functional $\R$ in $\mathcal{C}_{0,1}^{trunc}$ is attained in $\gamma_{trunc}^V$, where  the junction occurs at
\begin{equation}
V = v_0 + L,
\end{equation}
and the value of its resistance is 
\begin{equation}\label{rv}
\R[\gamma_{trunc}^{V}]=\Delta v-\frac{L}{2}.
\end{equation}
For $\gamma_{trunc}^V$, the slope of the loxodromic part is $k = 1$.
\end{theorem}

\begin{proof} Notice that $u'\geq 0$ in the domain of $\gamma_{truc}^{v_c}$ because $u_0<u_1$. Using the change of variables $y(v) = \Phi(u(v))$, the resistance functional for a truncated curve is the sum of the resistance of the loxodromic part   and the circular part. Thus, from \eqref{r5}, we have
\begin{equation*}
\R[v_c] = \int_{v_0}^{v_c} \frac{1}{1 + k^2}\,dv + \int_{v_c}^{v_1} 1\, dv = \frac{v_c - v_0}{1 + k^2} + (v_1 - v_c).
\end{equation*}
Substituting $k = L / (v_c - v_0)$ because of \eqref{k1}, the resistance takes the form 
\begin{equation*}
\R[v_c] = \frac{(v_c - v_0)^3}{(v_c - v_0)^2 + L^2} + v_1 - v_c.
\end{equation*}
To find the value $v_c$ that minimizes the total resistance, we differentiate $\R[v_c]$ with respect to $v_c$ and set it to zero. We have
\begin{equation*}
\R'[v_c] = -L^2\frac{ (v_c-v_0)^2-L^2}{\left(L^2+(v_c-v_0)^2\right)^2}. 
\end{equation*}
Setting $\R'[v_c]=0$ implies $v_c=v_0\pm L$. Since $v_c\in (v_0,v_1)$, and by the hypothesis $\Delta v > L$, the only valid critical point is $V=v_0+L$. It is easy to check that $\R''[V]=\frac{1}{2L}>0$, which proves that $\gamma_{trunc}^V$ determines 
  the global minimum of $\R[v_c]$ in the interval $(v_0,v_1)$. The computation of the resistance of $\gamma_{trunc}^{v_c}$ is straightforward yielding \eqref{rv}. Finally, from \eqref{k1}, we have $k=1$ at $V$.
 
\end{proof}

 \begin{corollary} For given points $A = (u_0, v_0)$ and $B = (u_1, v_1)$ in geodesic coordinates, with $u_0<u_1$ and $v_0<v_1$, if $\Delta v > \Phi(u_1)-\Phi(u_0)$, then the (smooth) loxodrome connecting $A$ and $B$   is not the global minimizer of the resistance in the class of functions $\mathcal{C}_{0,1}^m$
\end{corollary}

 \begin{figure}[h!t]
\centering
\includegraphics[width=0.55\linewidth]{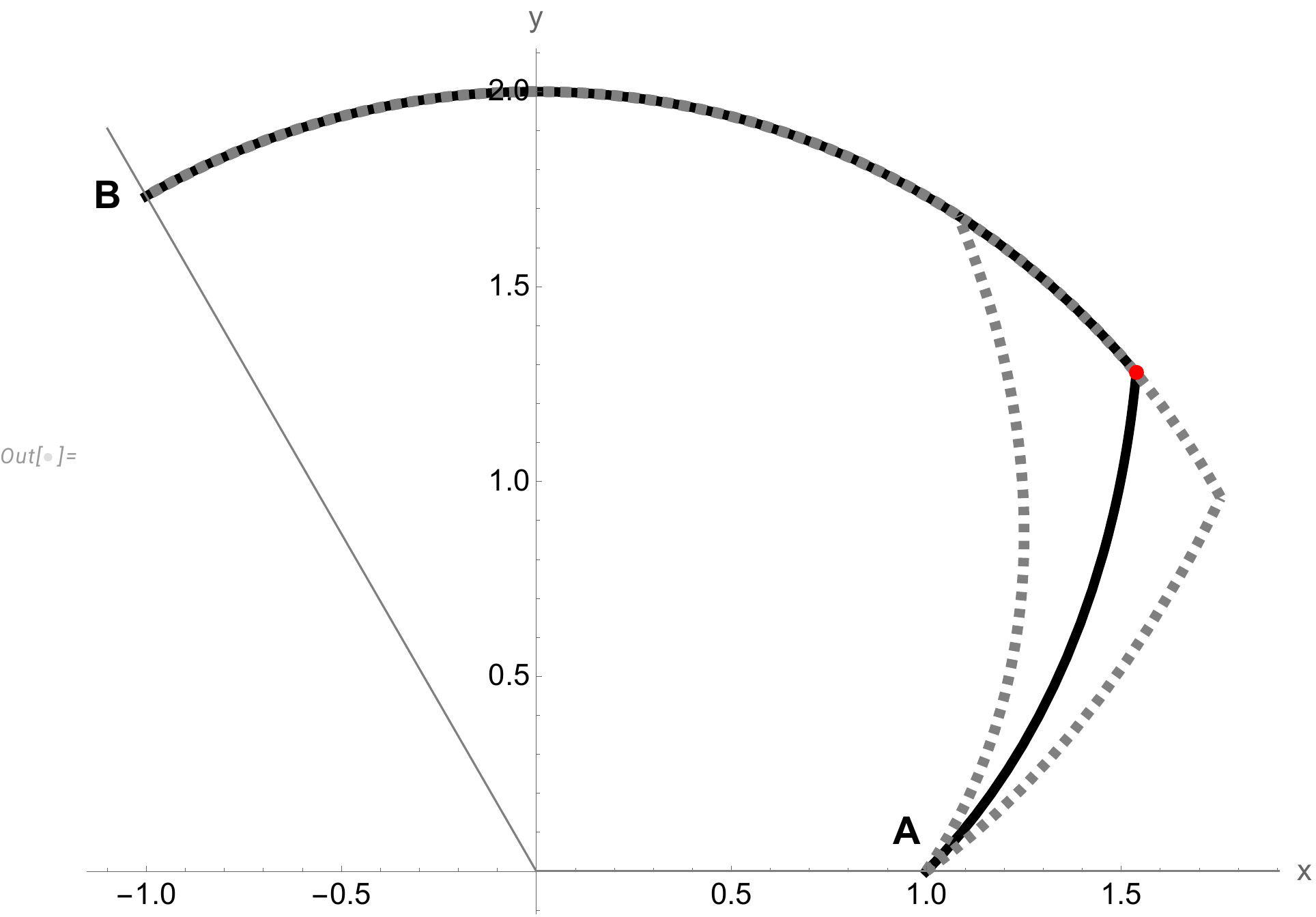}\quad \includegraphics[width=0.4\linewidth]{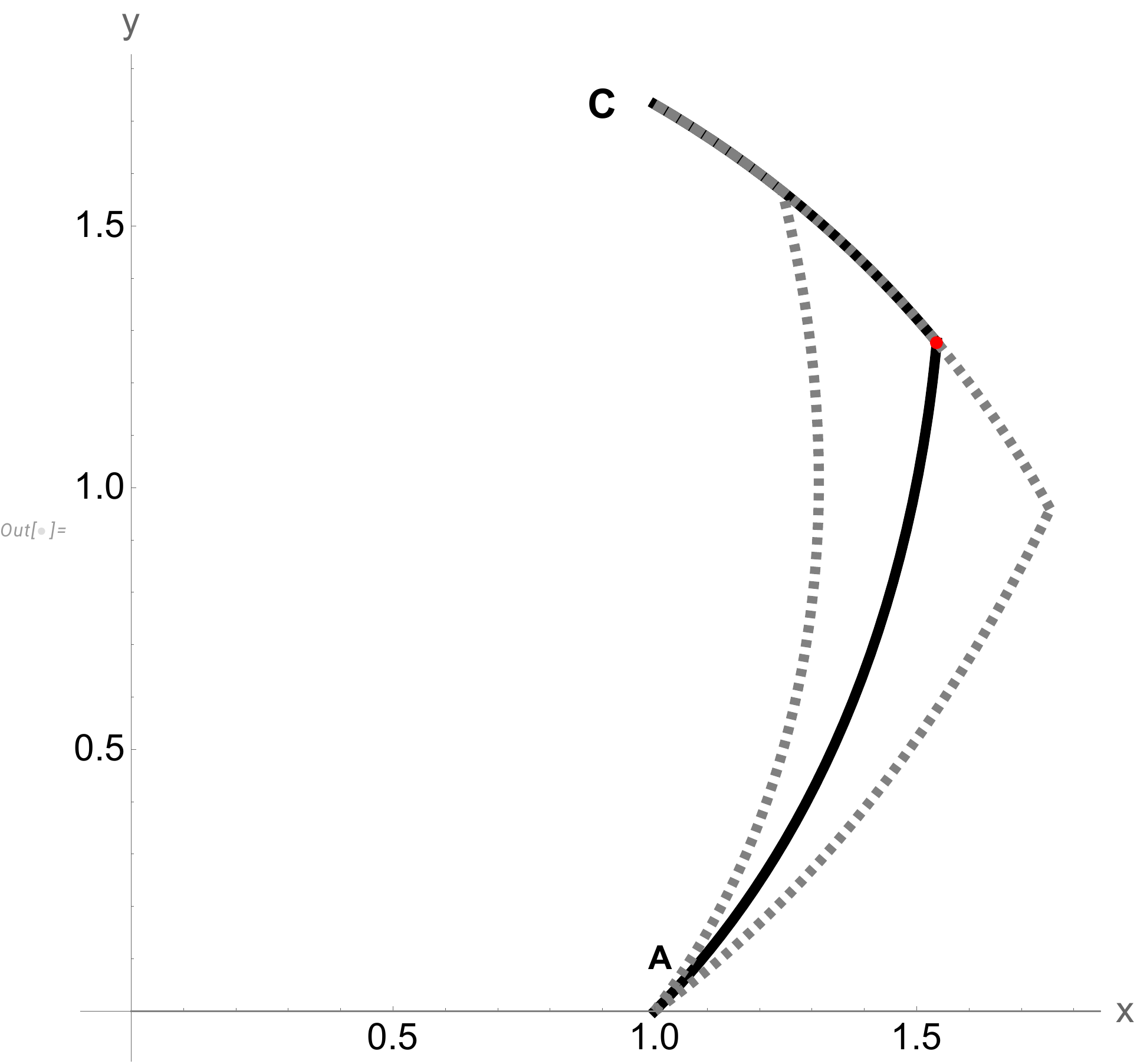}
\caption{Truncated loxodromes in $\r^2$. The endpoints in polar coordinates are $A=(1,0)$ and $B=(2,\frac{2\pi}{3})$ (left) and $C=(2,\pi/3)$ (right). Dashed truncated loxodromes for different values of $v_c$. The ideal truncated loxodrome (black) occurs when $V=L\approx 0.6931$.}
\label{fig2}
\end{figure}

 \begin{remark}
A consequence of Theorem \ref{t51} is that the optimal angular amplitude for the loxodromic part, $V -v_0=L$, depends only on the initial and final radii, $u_0$ and $u_1$, and it is independent of the angular amplitude of the domain $\Delta v$, provided the admissibility condition $\Delta v > L$ is satisfied. Physically, this implies that the optimal aerodynamic front achieves the required radial expansion (from $u_0$ to $u_1$) through a unique, fixed loxodrome. Once the loxodrome attains the radius $u_1$ at $v_c=V$, the optimal front is completed by  a circular arc of constant radius $u_1$ over the remaining interval $[V,v_1]$. 
\end{remark}

\section{Broken extremals and local minima}\label{s5}

In this section, we investigate the existence of broken extremals, that is, piecewise $C^1$-curves with corners, and their role as potential minimizers. We need to distinguish whether the corner occurs in the interior of the admissible set $\mathcal{C}_{0,1}^m$ because in such a case, it must satisfy the classical Weierstrass-Erdmann conditions. In contrast, if the corner occurs when $u'=0$, the situation is different. We begin by characterizing the case where the corners occur when $u'\not=0$.

\begin{theorem}\label{t41} 
Let $\gamma$ be a piecewise $C^1$ extremal for the resistance functional, parameterized as $u=u(v)$. Suppose $u'(v_0^-)=p$ and $u'(v_0^+)=q$ are non-zero.  Then $\gamma$ presents a corner at $v_0$ if and only if 
\begin{equation}\label{we}
p\cdot q = \frac{1}{3}f(u(v_0))^2.
\end{equation}
\end{theorem}

\begin{proof}
  The Lagrangian of \eqref{r11} is 
$\mathcal{L}=\frac{f^2}{ f^2+u'^2}$. By the Weierstrass-Erdmann corner conditions, the functions $\frac{\partial \mathcal{L}}{\partial u'}$ and $H = \mathcal{L} - u'\frac{\partial \mathcal{L}}{\partial u'}$ must be continuous at $v_0$. 
Computing  these quantities, we obtain
$$\frac{\partial \mathcal{L}}{\partial u'} = \frac{-2f(u)^2 u'}{(f(u)^2 + u'^2)^2}, \qquad H = \frac{f(u)^2 \left( f(u)^2 + 3u'^2 \right)}{(f(u)^2 + u'^2)^2}.$$
 The continuity conditions at $v_0$ require
$$\frac{-2f(u)^2 p}{(f(u)^2 + p^2)^2} = \frac{-2f(u)^2 q}{(f(u)^2 + q^2)^2}, \quad \mbox{and}\quad 
\frac{f(u)^2 \left( f(u)^2 + 3p^2 \right)}{(f(u)^2 + p^2)^2} = \frac{f(u)^2 \left( f(u)^2 + 3q^2 \right)}{(f(u)^2 + q^2)^2}.$$
Since $p, q \neq 0$, dividing the second equation by the first eliminates the denominators, leading to $ \frac{f(u)^2}{p} + 3p = \frac{f(u)^2}{q} + 3q,$ which simplifies to   \eqref{we}.
\end{proof}

As a consequence of this result, and because $f>0$, an internal corner occurs if both slopes have the same sign. We come back to the     admissible class  $\mathcal{C}_{0,1}^m$.

 \begin{corollary} \label{co52}
 Broken extremals   cannot be local minimizers of the resistance problem in interior of the  admissible class   $\mathcal{C}_{0,1}^m$. That is, if a minimizer satisfies $u'(v_2)>0$ at $v=v_2$ on both sides of a point $v_2$, then it must be smooth at $v_2$. 
\end{corollary}

\begin{proof}
Suppose  $\gamma \in \mathcal{C}_{0,1}^m$ is a minimizer with a corner at $v_2$. This requires $p\cdot q=\frac13 f(u(v_2))^2$. According to Proposition \ref{pr37},  for $\gamma$ to be a local minimizer, the smooth arcs must satisfy the Legendre necessary conditions $3p^2\geq f(u(v_2))^2$ and    $3q^2 \ge f(u(v_2))^2$. Multiplying these   yields 
$9 p^2 q^2 \ge f(u(v_2))^4$. The only way to satisfy this along with the condition $p\cdot q>0$ is that  $3p^2 = f(u(v_2))^2$ and $3q^2 = f(u(v_2))^2$. This implies $p = q$, contradicting the existence of a corner. 
\end{proof}

\begin{theorem}\label{t42}
Let $u(v) \in \mathcal{C}_{0,1}^m$ be an extremal connecting $A$ and $B$. If the slope of the extremal satisfies $u'(v) < f(u(v))$ at any point, then the variational problem admits no strong local minimizer among smooth extremals  in the class $\mathcal{C}_{0,1}^m$. 
\end{theorem}

\begin{proof}
We evaluate  the Weierstrass excess function $\mathcal{E}(u, p, q)$ defined by 
\begin{equation*}
\mathcal{E}(u, p, q) = \mathcal{L}(u, q) - \mathcal{L}(u, p) - (q - p)\frac{\partial \mathcal{L}}{\partial u'}(u, p).
\end{equation*}
For a strong minimum, we need  $\mathcal{E}\geq 0$ for all test slopes $p,q\geq 0$.  Computing $\mathcal{E}$, we obtain
\begin{equation*}
\begin{split}
\mathcal{E}(u, p, q) &= \frac{f(u)^2}{f(u)^2 + q^2} - \frac{f(u)^2}{f(u)^2 + p^2} + (q - p)\frac{2f(u)^2 p}{(f(u)^2 + p^2)^2}\\
&= \frac{f(u)^2 (p - q)^2}{(f(u)^2 + q^2)(f(u)^2 + p^2)^2} \Big( p^2 + 2pq - f(u)^2 \Big).
\end{split}
\end{equation*}
The sign of $\mathcal{E}(u, p, q)$ depends on the polynomial $S(p,q) = 2pq + p^2 - f(u)^2$. At $q = 0$, we have $S(p,0) = p^2 - f(u)^2$. By hypothesis, we have  $S(p,0)<0$. Consequently, $S(p,q)<0$ if $q$ is sufficiently close to zero, proving that  the extremal does not satisfy the Weierstrass condition.
\end{proof} 
 
 This theorem and Corollary \ref{co52} suggest that if a (smooth) loxodrome  fails to be a strong local minimizer, it is because its slope $u'$ falls below the value $f(u)$.    As we shall see in Section \ref{s6}, this lack of strong minimality for low  slopes   justifies why we focus on truncated loxodromes as candidates to the global minimizer when the optimal solution is not smooth.

To establish sufficiency for a weak local minimum, we must verify the Jacobi condition, which asserts that the extremal must not contain any points conjugate to the initial point.  

\begin{lemma} 
Let $A= (u_0, v_0)$ be a fixed initial point in $\Sigma$ and consider the family of loxodromic extremals $\mathcal{F} = \{ \gamma_k \}_{k \in \mathbb{R}^+}$ starting from $A$ and where $\gamma_k$ are defined in geodesic coordinates $w_k=w_k(u)$ by 
\begin{equation} \label{field_eq}
w_k(u) = v_0 + \frac{1}{k} G(u), \quad \text{where} \quad G(u) = \int_{u_0}^{u} \frac{ds}{f(s)}.
\end{equation}
Then, for any $u > u_0$, the family $\mathcal{F}$ constitutes a central field of extremals covering the domain $\mathcal{D} = \{ (u, v) \in \Sigma : u > u_0, v > v_0 \}$.
\end{lemma}

\begin{proof}
To prove that $\mathcal{F}$ forms a central field, we must verify that all $\gamma_k$ cover $\mathcal{D}$ without intersecting (existence and uniqueness), and that the coordinate transformation $(u, k) \mapsto (u, v)$ is smooth and non-degenerate. Notice that $\mathcal{D}$ is simply connected because it is a quadrant of $\r^2$. 

First, we show that for any   $B=(u_1, v_1) \in \mathcal{D}$, there is exactly one extremal in $\mathcal{F}$ connecting $A$ to $B$. From \eqref{field_eq}, we have to solve   the equation $v_1 =w_k(u_1)= v_0 + \frac{1}{k} G(u_1)$. But this gives a unique solution $k = \frac{G(u_1)}{v_1 - v_0}$, which is positive. Thus, we have proved the existence and uniqueness of extremals. 

The non-degeneracy of  the mapping $(u, k) \mapsto (u, v)$ defined by \eqref{field_eq} is verified by computing its Jacobian matrix, which is  
\begin{equation*}
\frac{\partial (u, v)}{\partial (u, k)} = 
\begin{pmatrix} 
1 & 0 \\ 
\frac{1}{k f(u)} & -\frac{1}{k^2} G(u) 
\end{pmatrix}.
\end{equation*}
Since its determinant is non-zero in $\mathcal{D}$, the inverse function theorem ensures that the mapping is a local diffeomorphism, and the extremals form a smooth field. 
\end{proof}
 
 Having established that the family of loxodromes forms a valid central field 
 over the domain $\mathcal{D}$, we can now address the Jacobi necessary condition. 
 Since the central field guarantees the absence of an envelope, there are no conjugate points. This proves the following   sufficiency theorem.
 
 \begin{theorem} \label{t43}
Let $\gamma$ be a loxodrome connecting two points in $\mathcal{D}$. If $\gamma$ satisfies the strict Legendre condition $3u'^2>f(u)^2$, then   $\gamma$ is a weak local minimizer for the resistance functional.
\end{theorem}

\begin{proof}
By the previous lemma, the loxodrome $\gamma$ is embedded in a central field of extremals $\mathcal{F}$ that   covers the simply connected domain $\mathcal{D}$ without forming an envelope. According to Jacobi's envelope theorem (see, e.g., \cite{gf}), the absence of such an envelope guarantees that the loxodrome contains no conjugate points relative to the initial point. Because $\gamma$ satisfies the strict Legendre condition and is free of conjugate points, it fulfills all classical sufficient conditions to provide a weak local minimum for the resistance functional.
\end{proof}

\begin{corollary}\label{co44}
Let $\gamma$ be a regular loxodromic extremal connecting two points $A$ and $B$. If the strict Legendre condition $3u'^2 > f(u)^2$ holds along $\gamma$, then $\gamma$ provides a strict weak local minimizer for the resistance functional, but it fails to be a strong local minimizer.
\end{corollary}

\begin{proof}
The first statement holds because  the extremal satisfies the Legendre condition (Proposition \ref{pr37}) and the Jacobi condition (Theorem \ref{t43}). The second statement is a consequence of Theorem \ref{t42}.
\end{proof}

 \section{Global minimizers in the admissible class of monotonic curves}\label{s6}
 
 In this section, we determine the global minimizers of the resistance problem. The answer depends on     the angular amplitude between the two points. The   behavior of the minimizer is governed by the convexity of the integrand $g(p) =\frac{1}{1+p^2}$ in the expression \eqref{16} for the resistance.  Since $g$ is not convex for $p < 1$, smooth extremals with low slopes (less than $1$) cannot be global minimizers.  First, we consider the case where the  amplitude is small relative to the difference in the radii.  
  
 \begin{theorem} \label{t61}
 Let  $A=(u_0, v_0)$ and $B=(u_1, v_1)$ be two points with $u_0 < u_1$ and $v_0<v_1$. If the angular amplitude $\Delta v = v_1 - v_0$ satisfies 
\begin{equation}
\Delta v \le \Phi(u_1) - \Phi(u_0),
\end{equation}
then the unique global minimizer of $\mathcal{R}$ in the class of admissible monotonic curves $\mathcal{C}_{0,1}^{m}$ is the loxodrome $\gamma_{lox}$ defined in geodesic coordinates by 
$$\Phi(u(v)) = k(v-v_0)+\Phi(u_0),$$
 where 
$$k = \frac{\Phi(u_1) - \Phi(u_0)}{\Delta v}.$$
\end{theorem}

\begin{proof}
We introduce the change of variables $y(v) = \Phi(u(v))$. Notice that the constraint $u' \ge 0$ is preserved as $y' \ge 0$. The resistance functional becomes 
\begin{equation}
\mathcal{R}[y] = \int_{v_0}^{v_1} \frac{1}{1 + y'^2}\, dv.
\end{equation}
The curve $\gamma_{lox}$ is, by definition, an extremal and it is immediate that $\gamma_{lox}$ connects $A$ to $B$ because $\Phi(u(v_0))=\Phi(u_0)$ and $\Phi(u(v_1))=k\Delta v+\Phi(u_0)=\Phi(u_1)$ by the definition of $k$. 

Let $g(p) = \frac{1}{1+p^2}$ be the integrand of $\mathcal{R}[y]$. To find the global minimum, we consider the lower convex envelope of $g(p)$ on the domain $p \in [0, \infty)$, denoted by $g^{**}(p)$. Since $g$ is concave on $[0,1)$ and convex on $(1,\infty)$, the function $g^{**}$ in the interval $[0,1]$ is the tangent line to $g$ from $(0, g(0)) = (0, 1)$ to $(1,g(1))=(1,\frac12)$; on $[1,\infty)$, we have $g^{**}=g$. This gives
\begin{equation}
g^{**}(p) = \begin{cases} 
1 - \frac{1}{2}p & \text{if } 0 \le p \le 1 \\
\frac{1}{1+p^2} & \text{if } p \geq 1 
\end{cases}
\end{equation}
Let $u\in \mathcal{C}_{0,1}^m$. Since $g\geq g^{**} $ and $g^{**}$ is convex, we can apply Jensen's inequality obtaining
\begin{equation}
 \begin{split}
\R[u]&= \int_{v_0}^{v_1} g(y')\, dv \ge \int_{v_0}^{v_1} g^{**}(y')\, dv \ge \Delta v\cdot g^{**} \left( \frac{1}{\Delta v} \int_{v_0}^{v_1} y'\, dv \right) \\
&= \Delta v\cdot g^{**} \left( \frac{1}{\Delta v} (\Phi(u_1)-\Phi(u_0)) \right)= \Delta v\cdot g^{**}(k)\\
&=\Delta v\cdot g(k),
\end{split}
\end{equation}
where in the last equality, we used the fact that $k\geq 1$ and the definition of $g^{**}$. Furthermore, the resistance of the loxodrome is given in \eqref{r5}, which yields $\R[\gamma_{lox}]=\frac{\Delta v}{1+k^2}=\Delta v\cdot g(k)$. Thus, $\R[u]\geq \R[\gamma_{lox}]$. 

To establish uniqueness, assume there is another curve $u=u(v)$ whose resistance is $\Delta v\cdot g(k)$. Then we have equality throughout the chain of inequalities. In particular, equality in Jensen's inequality holds if and only if the argument $y'(v)$ is constant almost everywhere. Thus $y'(v) = k$ identically and this proves that $u(v)$ is a loxodrome. By the boundary conditions at $v_0$ and $v_1$, $u(v)$ must be $\gamma_{lox}$. 
\end{proof}

We now address the complementary case where the angular domain is large, that is, 
$\Delta v > \Phi(u_1) - \Phi(u_0)$. Under  this assumption,  a smooth loxodrome connecting 
$A$ and $B$ would necessarily have a slope $k<1$, falling into the region where the integrand 
$g(p)$ is non-convex. As shown in  Theorem \ref{t42}, such an extremal fails the Weierstrass 
condition and cannot be a strong  local minimum. This suggests  the need to truncate the loxodrome  by fixing the slope to be $k=1$ and completing the path with a circular arc ($u'=0$). This leads to   the 
truncated loxodrome $\gamma_{trunc}^V$ defined in Section \ref{s4}, which we prove to be  a global minimizer.  

\begin{theorem}\label{t62}
Let $A=(u_0, v_0)$ and $B=(u_1, v_1)$ be two points with $u_0 < u_1$ and $v_0<v_1$. If the angular amplitude $\Delta v = v_1 - v_0$ satisfies 
\begin{equation}
\Delta v > \Phi(u_1) - \Phi(u_0),
\end{equation}
then the  minimum value of the resistance functional $\mathcal{R}$ in the class of admissible monotonic curves $\mathcal{C}_{0,1}^{m}$ is given by $\Delta v - \frac{1}{2}(\Phi(u_1) - \Phi(u_0))$. This minimum value is attained by the truncated loxodrome $\gamma_{trunc}^V$, with junction point $V = v_0 + \Phi(u_1) - \Phi(u_0)$.
\end{theorem}

\begin{proof}
Let $L = \Phi(u_1) - \Phi(u_0)$. We use the change of variables $y(v) = \Phi(u(v))$ and the lower convex envelope $g^{**}(p)$ of the integrand $g(p) = \frac{1}{1+p^2}$ introduced in Theorem \ref{t61}. Since $\Delta v > L$, we have $k = \frac{L}{\Delta v} < 1$. 
Recall that in  $[0, 1]$, the convex envelope is given by the straight line $g^{**}(p) = 1 - \frac{1}{2}p$. For any admissible curve $u \in \mathcal{C}_{0,1}^m$, using $g\geq g^{**}$, Jensen's inequality yields
\begin{equation}
\begin{split}
\mathcal{R}[u] &= \int_{v_0}^{v_1} g(y')\,dv \ge \int_{v_0}^{v_1} g^{**}(y')\, dv \ge \Delta v \cdot g^{**}\left( \frac{1}{\Delta v} \int_{v_0}^{v_1} y'\, dv \right) \\
&= \Delta v \cdot g^{**}(k) = \Delta v \left( 1 - \frac{1}{2}k \right) = \Delta v - \frac{1}{2}L.
\end{split}
\end{equation}
This establishes an absolute lower bound for the resistance. Furthermore, we know from \eqref{rv} in Theorem \ref{t51}, that $
\mathcal{R}[\gamma_{trunc}^V] =  \Delta v - \frac{1}{2}L$. This shows that  $\gamma_{trunc}^V$  is a global minimizer in $\mathcal{C}_{0,1}^m$. 
\end{proof}

\begin{remark}
Unlike the case where $\Delta v \le \Phi(u_1) - \Phi(u_0)$, the global minimizer here is not uniquely defined in the class $\mathcal{C}_{0,1}^m$.   Any monotonic piecewise $C^1$-curve whose derivative $y'(v)$ alternates between $1$ and $0$, maintaining the total measure of the sets where $y'=1$ equal to $L$, will yield the same  resistance (there is a similar discussion of non-uniqueness in \cite{st} for the classical two-dimensional Newton problem). However, $\gamma_{trunc}^V$ is the natural physical choice because there is only one junction.
\end{remark}

The substitution $y(v) = \log u(v)$ reveals the following connection to  the classical two-dimensional Newton's problem in Cartesian coordinates under a parallel flow.    Under this transformation, loxodromes of the surfaces in geodesic coordinates map   to straight lines of $\r^2$ in  Cartesian coordinates. Furthermore, the critical slope   $k=1$   corresponds to  the   $45^\circ$ critical angle in Newton's problem  analyzed by Silva and Torres \cite{st}. While they established the absolute minimality of the constant-slope profile for $k > 1$ using optimal control theory and the Pontryagin maximum principle, we provide an alternative   proof based on the lower convex envelope and Jensen's inequality.  

\section*{Declarations }

The author has no conflicts to disclose. No data was used for the research described in the article.


 \section*{Acknowledgements}
Rafael L\'opez has been partially supported by MINECO/MICINN/FEDER grant no. PID2023-150727NB-I00, and by the ``Mar\'{\i}a de Maeztu'' Excellence Unit IMAG, reference CEX2020-001105- M, funded by MCINN/AEI/10.13039/ 501100011033/ CEX2020-001105-M.


\end{document}